\documentclass[a4paper,twoside,11pt]{article}
\usepackage{a4wide, fancyhdr, amsmath, amssymb, mathtools, yfonts}
\usepackage{graphicx}
\usepackage{tikz}
\usepackage[all]{xy}
\usepackage[utf8]{inputenc}
\usepackage{amsthm}
\usepackage[english]{babel}
\usepackage{chngcntr}
\usepackage{ifthen}
\usepackage{calc}
\usepackage{hyperref}
\numberwithin{equation}{section}

\newcommand{\Z}{\mathbb{Z}}
\newcommand{\Q}{\mathbb{Q}}

\newcommand{\OO}{\mathcal{O}}
\newcommand{\CL}{\mathrm{Cl}}

\newcommand\ZZP{\mathbb{Z}[\sqrt{-2p}]}
\newcommand\QQP{\mathbb{Q}(\sqrt{-2p})}

\newcommand\FF{\mathbb{F}}
\newcommand\MM{\mathbb{M}}
\newcommand\DD{\mathcal{D}}

\newcommand\aaa{\mathfrak{a}}

\newcommand\mm{\mathfrak{m}}

\newcommand\nn{\mathfrak{n}}

\newcommand\pp{\mathfrak{p}}

\newcommand\qq{\mathfrak{q}}

\newcommand\rk{\mathrm{rk}}

\newcommand\Norm{\mathrm{N}}

\newcommand\rat{\mathrm{r}}

\newcommand\ve{\varepsilon}

\newtheorem{theorem}{Theorem}
\newtheorem{lemma}{Lemma}[section]
\newtheorem{prop}[lemma]{Proposition}
\newtheorem{corollary}[theorem]{Corollary}

\title{\vspace{-\baselineskip}\sffamily\bfseries On the $16$-rank of class groups of $\Q(\sqrt{-2p})$ for primes $p \equiv 1 \bmod 4$}
\author{Peter Koymans, Djordjo Milovic
\\{\tt p.h.koymans@math.leidenuniv.nl}
\\{\tt dmilovic@math.ias.edu}}

\date{\today}

\begin{document}
\maketitle

\begin{abstract}
We use Vinogradov's method to prove equidistribution of a spin symbol governing the $16$-rank of class groups of quadratic number fields $\QQP$, where $p \equiv 1 \bmod 4$ is a prime.
\end{abstract}

\section{Introduction}
Recently, the authors have used Vinogradov's method to prove density results about elements of order $16$ in class groups in certain \textit{thin} families of quadratic number fields parametrized by a single prime number, namely the families $\{\Q(\sqrt{-2p})\}_{p\equiv -1\bmod 4}$ and $\{\Q(\sqrt{-p})\}_{p}$ \cite{Milovic2, KM1}. In this paper, we establish a density result for the family $\{\Q(\sqrt{-2p})\}_{p\equiv 1\bmod 4}$, thereby completing the picture for the $16$-rank in thin families of imaginary quadratic fields of even discriminant. Although our overarching methods are similar to those originally developed in the work of Friedlander et al.\ \cite{FIMR}, the technical difficulties in the present case are different and require a more careful study of the spin symbols governing the $16$-rank. The main distinguishing feature of the present work is that this careful study allows us to avoid relying on a conjecture about short character sums appearing in \cite{FIMR, KM1}, thus making our results unconditional.

More generally, given a sequence of complex numbers $\{a_n\}_{n}$ indexed by natural numbers, a problem of interest in analytic number theory is to prove an asymptotic formula for the sum over primes 
$$
S(X):=\sum_{\substack{p\text{ prime} \\ p\leq X}}a_p
$$
as $X\rightarrow \infty$. Many sequences $\{a_n\}_n$ admit asymptotic formulas for $S(X)$ via various generalizations of the Prime Number Theorem, with essentially the best known error terms coming from ideas of de la Val\'{e}e Poussin already in 1899 \cite{LVP}. In 1947, Vinogradov \cite{Vino1, Vino2} invented another method to treat certain sequences which could not be handled with a variant of the Prime Number Theorem. His method has since been clarified and made easier to apply, most notably by Vaughan \cite{Vaughan} and, for applications relating to more general number fields, by Friedlander et al.\ \cite{FIMR}. Nonetheless, there is a relative paucity of interesting sequences $\{a_n\}_n$ that admit an asymptotic formula for $S(X)$ via Vinogradov's method. The purpose of this paper is to present yet another such sequence, of a similar nature as those appearing in \cite{FIMR, KM1}; similarly as in \cite{KM1}, the asymptotics we obtain have implications in the arithmetic statistics of class groups of number fields.

Let $p\equiv 1\bmod 4$ be a prime number, and let $\CL(-8p)$ denote the class group of the quadratic number field $\QQP$ of discriminant $-8p$. The finite abelian group $\CL(-8p)$ measures the failure of unique factorization in the ring $\ZZP$. By Gauss's genus theory \cite{Gauss}, the $2$-part of $\CL(-8p)$ is cyclic and non-trivial, and hence determined by the largest power of $2$ dividing the order of $\CL(-8p)$. For each integer $k\geq 1$, we define a density $\delta(2^k)$, if it exists, as
$$
\delta(2^k) := \lim_{X\rightarrow\infty}\frac{\#\{p\leq X:\ p\equiv 1\bmod 4,\ 2^k|\#\CL(-8p)\}}{\#\{p\leq X:\ p\equiv 1\bmod 4\}}.
$$
As stated above, the $2$-part of $\CL(-8p)$ is cyclic and non-trivial, so $\delta(2) = 1$. It follows from the Chebotarev Density Theorem (a generalization of the Prime Number Theorem) that $\delta(4) = \frac{1}{2}$ and $\delta(8) = \frac{1}{4}$; indeed, R\'{e}dei \cite{Redei} proved that $4|\#\CL(-8p)$ if and only if $p$ splits completely in $\Q(\zeta_8)$, and Stevenhagen \cite{Ste2} proved that $8|\#\CL(-8p)$ if and only if $p$ splits completely in $\Q(\zeta_8, \sqrt[4]{2})$, where $\zeta_8$ denotes a primitive $8$th root of unity. The qualitative behavior of divisibility by $16$ departs from that of divisibility by lower 2-powers in that it can no longer be proved by a simple application of the Chebotarev Density Theorem. We instead use Vinogradov's method to prove    

\begin{theorem}\label{mainThm}
For a prime number $p\equiv 1\bmod 4$, let $e_p = 0$ if $\CL(-8p)$ does not have an element of order~$8$, let $e_p = 1$ if $\CL(-8p)$ has an element of order~$16$, and let $e_p = -1$ otherwise. Then for all $X > 0$, we have
$$
\sum_{\substack{p\leq X \\ p\equiv 1\bmod 4}}e_p\ll X^{1-\frac{1}{3200}},
$$
where the implied constant is absolute. In particular, $\delta(16) = \frac{1}{8}$.
\end{theorem}
In combination with \cite{Milovic2}, we get
\begin{corollary}\label{mainCor2}
For a prime number $p$, let $h_2(-2p)$ denote the cardinality of the $2$-part of the class group $\CL(-8p)$. For an integer $k\geq 0$, let $\delta'(2^k)$ denote the natural density (in the set of all primes) of primes $p$ such that $h_2(-2p) = 2^k$, if it exists. Then $\delta'(1) = 0$, $\delta'(2) =\frac{1}{2}$, $\delta'(4) =\frac{1}{4}$, and $\delta'(8) =\frac{1}{8}$.  
\end{corollary}
The power-saving bound in Theorem~\ref{mainThm}, similarly to the main results in \cite{Milovic2} and \cite{KM1}, is another piece of evidence that \textit{governing fields} for the $16$-rank do \textit{not} exist. For a sampling on previous work about governing fields, see \cite{CohnLag}, \cite{CohnLag2}, \cite{Morton}, and \cite{Ste1}.
  
The strategy to prove Theorem~\ref{mainThm} is to construct a sequence $\{a_n\}_n$ which simultaneously carries arithmetic information about divisibility by $16$ when $n$ is a prime number congruent to $1$ modulo $4$ and is conducive to Vinogradov's method. On one hand, the criterion for divisibility by $16$ cannot be stated naturally over the rational numbers $\Q$. For instance, even the criterion for divisibility by $8$ is most naturally stated over a field of degree $8$ over $\Q$. On the other hand, proving analytic estimates in a number field generally becomes more difficult as the degree of the number field increases, as exemplified by the reliance on a conjecture on short character sums in \cite{FIMR}. We manage to work over $\Q(\zeta_8)$, a field of degree $4$. Although the methods of Friedlander et al.\ \cite{FIMR} narrowly miss the mark of being unconditional for number fields of degree $4$, we manage to exploit the arithmetic structure of our sequence to ensure that Theorem~\ref{mainThm} is unconditional.     

Lastly, for work concerning the average behavior of the $2$-parts of class groups of quadratic number fields in families that are \textit{not} thin, i.e., for which the average number of primes dividing the discriminant grows as the discriminant grows, we point the reader to the extensive work of Fouvry and Kl\"{u}ners \cite{FK1, FK2, FK3, FK4} on the $4$-rank and certain cases of the $8$-rank and more recently to the work of Smith on the $8$- and higher $2$-power-ranks \cite{Smith1, Smith2}. While Smith's methods in \cite{Smith2} appear to be very powerful, the authors believe that they are unlikely to be applicable to thin families of the type appearing in this paper.   

\subsection*{Acknowledgements}
The authors would like to thank Jan-Hendrik Evertse and Carlo Pagano for useful discussions. The second author was supported by National Science Foundation agreement No.\ DMS-1128155. 

\section{Encoding the $16$-rank of $\CL(-8p)$ into sequences $\{a_{\nn}\}_{\nn}$}
Given an integer $k\geq 1$, the $2^k$-rank of a finite abelian group $G$, denoted by $\rk_{2^k}G$, is defined as the dimension of the $\FF_2$-vector space $2^{k-1}G/2^kG$. If the $2$-part of $G$ is cyclic, then $\rk_{2^k}G\in\{0, 1\}$, and $\rk_{2^k}G = 1$ if and only if $2^k|\#G$. The order of a class group is called the class number, and we denote the class number of $\CL(-8p)$ by $h(-8p)$.

The criterion for divisibility of $h(-8p)$ by $16$ that we will use is due to Leonard and Williams \cite[Theorem 2, p.\ 204]{LW82}. Given a prime number $p\equiv 1\bmod 8$ (so that $4|h(-8p)$), there exist integers $u$ and $v$ such that
\begin{equation}\label{puv}
p = u^2-2v^2, \ \ \ \ \ u>0.
\end{equation}
The integers $u$ and $v$ are \textit{not} uniquely determined by $p$; nevertheless, if $(u_0, v_0)$ is one such pair, then, every such pair $(u, v)$ is of the form $u+v\sqrt{2} = \ve^{2m}(u_0\pm v_0\sqrt{2})$ for some $m\in\Z$, where $\ve = 1+\sqrt{2}$. The criterion for divisibility by $8$ can be restated in terms of a quadratic residue symbol; one has 
$$
8|h(-8p)\Longleftrightarrow \left(\frac{u}{p}\right)_2 = 1.
$$
Note that $1 = (u/p)_2 = (p/u)_2 = (-2/u)_2$, so that $8|h(-8p)$ if and only if $u\equiv 1, 3\bmod 8$. As $\ve^2(u+v\sqrt{2}) = (3u+4v) + (2u+3v)\sqrt{2}$ and $v$ is even, we can always choose $u$ and $v$ in \eqref{puv} so that $u\equiv 1\bmod 8$. The criterion for divisibility of $h(-8p)$ by $16$ states that if $u$ and $v$ are integers satisfying \eqref{puv} and $u \equiv 1\bmod 8$, then
$$
16|h(-8p)\Longleftrightarrow \left(\frac{u}{p}\right)_4 = 1,
$$ 
where $(u/p)_4$ is equal to $1$ or $-1$ depending on whether or not $u$ is a fourth power modulo~$p$. To take advantage of the multiplicative properties of the fourth-power residue symbol, one has to work over a field containing $i = \sqrt{-1}$, a primitive fourth root of unity. Since $u$ is naturally defined via the splitting of $p$ in $\Q(\sqrt{2})$, we see that the natural setting for the criterion above is the number field
$$
M := \Q(\sqrt{2}, i) = \Q(\zeta_8),
$$
of degree $4$ over $\Q$. It is straightforward to check that the class number of $M$ and each of its subfields is $1$, that $2$ is totally ramified in $M$, and that the unit group of its ring of integers $\OO_M = \Z[\zeta_8]$ is generated by $\zeta_8$ and $\ve = 1 + \sqrt{2}$. Note that $M/\Q$ is a normal extension with Galois group isomorphic to the Klein four group, say $\{1, \sigma, \tau, \sigma\tau\}$, where $\sigma$ fixes $\Q(i)$ and $\tau$ fixes $\Q(\sqrt{2})$.
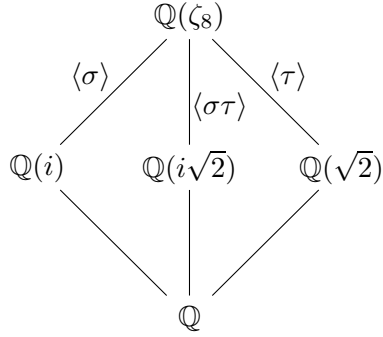
\begin{center}
\begin{tikzpicture}
  \draw (0, 0) node[]{$\Q$};
  \draw (0, 2) node[]{$\Q(i\sqrt{2})$};
	\draw (-2, 2) node[]{$\Q(i)$};
	\draw (2, 2) node[]{$\Q(\sqrt{2})$};
  \draw (0, 4) node[]{$\Q(\zeta_8)$};
	\draw (0, 0.3) -- (0, 1.7);
	\draw (0, 2.3) -- (0, 3.7);
	\draw (0.3, 0.3) -- (1.7, 1.7);
	\draw (-0.3, 0.3) -- (-1.7, 1.7);
	\draw (1.7, 2.3) -- (0.3, 3.7);
	\draw (-1.7, 2.3) -- (-0.3, 3.7);
	\draw (1.3, 3.2) node[]{$\left\langle \tau\right\rangle$};
	\draw (0.4, 2.8) node[]{$\left\langle \sigma\tau \right\rangle$};
	\draw (-1.3, 3.2) node[]{$\left\langle \sigma\right\rangle$};
\end{tikzpicture}
\end{center} 
Let $p\equiv 1\bmod 8$ be a prime, so that $p$ splits completely in $M$. Then there exists $w\in\OO_M$ such that $\Norm(w) = p$, i.e., such that $p = w\sigma(w)\tau(w)\sigma\tau(w)$. Note that the inclusion $\Z\hookrightarrow \OO_M$ induces an isomorphism $\Z/(p)\cong\OO_M/(w)$, so that an integer $n$ is a fourth power modulo $p$ exactly when it is a fourth power modulo $w$. As $w\tau(w)\in \Z[\sqrt{2}]$, there exist integers $u$ and $v$ such that $w\tau(w) = u+v\sqrt{2}$. Then $u = (w\tau(w) + \sigma(w)\sigma\tau(w))/2$. With this in mind, we define, for any $\alpha\in\Z[\sqrt{2}]$,
$$
\rat(\alpha) = \frac{1}{2}\left(\alpha + \sigma(\alpha)\right)
$$
and, for \textit{any} \textit{odd} (i.e., coprime to $2$) $w\in\OO_M$, not necessarily prime,
$$
[w] := \left(\frac{\rat(w\tau(w))}{w}\right)_4,
$$
where $(\cdot/\cdot)_4$ is the quartic residue symbol in $M$; we recall the definition of $(\cdot/\cdot)_4$ in the next section. A simple computation shows that $\rat(w\tau(w)) > 0$ for any non-zero $w\in\OO_M$. Hence $16|h(-8p)$ if and only if $[w] = 1$, where $w$ is any element of $\OO_M$ such that $\Norm(w) = p$ and $\rat(w)\equiv 1\bmod 8$.

Given a Dirichlet character $\chi$ modulo $8$, we define, for any odd $w\in\OO_M$, 
$$
[w]_{\chi} := [w]\cdot \chi(\rat(w\tau(w))).
$$
Then
$$
\frac{1}{4}\sum_{\chi\bmod 8}[w]_{\chi} = 
\begin{cases}
[w] & \text{ if }\rat(w\tau(w)) \equiv 1\bmod 8, \\
0 & \text{ otherwise,}
\end{cases}
$$
where the sum is over all Dirichlet characters modulo $8$. Another simple computation shows that, for all odd $w\in\OO_M$, we have $[\zeta_8w] = [w]$. We note that $\rat(\ve^2 \alpha)\equiv 3\cdot \rat(\alpha)\bmod 8$ for any $\alpha\in\Z[\sqrt{2}]$, so that $\chi(\rat(\ve^2 w \tau(\ve^2 w))) = \chi(\rat(w\tau(w)))$ for every Dirichlet character $\chi$ modulo $8$. Finally, we note that
\begin{equation}\label{bracket2}
[w] = \left(\frac{16\rat(w\tau(w))}{w}\right)_4 = \left(\frac{8\sigma(w)\sigma\tau(w)}{w}\right)_4,
\end{equation}
so that 
$$
[\ve w] = \left(\frac{\sigma(\ve)}{w}\right)_2[w],
$$
and hence $[\ve^2 w] = [w]$. Having determined the action of the units $\OO_M^{\times}$ on $[\cdot]_{\chi}$, we can define, for each Dirichlet character $\chi$ modulo $8$, a sequence $\{a(\chi)_{\nn}\}_{\nn}$ indexed by \textit{ideals} of $\OO_M$ by setting $a(\chi)_{\nn} = 0$ if $\nn$ is even, and otherwise 
\begin{equation}\label{defachin}
a(\chi)_{\nn}:= [w]_{\chi} + [\ve w]_{\chi},
\end{equation}
where $w$ is any generator of the odd ideal $\nn$. Again because $\rat(\ve^2 \alpha)\equiv 3\cdot \rat(\alpha)\bmod 8$ for any $\alpha\in\Z[\sqrt{2}]$, we see that if $8|h(-8p)$, then exactly one of $\rat(w\tau(w))$ and $\rat(\ve w\tau(\ve w))$ is $1\bmod 8$, and if $8\nmid h(-8p)$, then neither is $1\bmod 8$. We have proved
\begin{prop}\label{seqCrit}
Let $p\equiv 1\bmod 8$ be a prime, and let $\pp$ be a prime ideal of $\OO_M$ lying above~$p$. Then
$$
\frac{1}{4}\sum_{\chi\bmod 8}a(\chi)_{\pp} =
\begin{cases}
1 & \text{ if }16|h(-8p), \\
-1 & \text{ if }8|h(-8p)\text{ but }16\nmid h(-8p), \\
0 & \text{otherwise,}
\end{cases}
$$
where the sum is over Dirichlet characters modulo $8$.
\end{prop} 

\section{Prerequisites}\label{sQR}
We now collect some definitions and facts that we will use in our proof of Theorem~\ref{mainThm}.

\subsection{Quartic residue symbols and quartic reciprocity}
Let $L$ be a number field with ring of integers $\OO_L$. Let $\pp$ be an odd prime ideal of $\OO_L$ and let $\alpha\in\OO_L$. One defines the \textit{quadratic residue symbol} $\left(\alpha/\pp\right)_{L, 2}$ by setting
\[
\left(\frac{\alpha}{\pp}\right)_{L, 2} := 
\begin{cases}
0 & \text{if }\alpha\in\pp \\
1 & \text{if }\alpha\notin\pp\text{ and }\alpha\equiv \beta^2\bmod \pp\text{ for some }\beta\in\OO_L \\
-1 & \text{otherwise.}
\end{cases}
\]
Then we have $(\alpha/\pp)_{L, 2} \equiv \alpha^{\frac{\Norm_{L/\Q}(\pp) - 1}{2}} \bmod \pp$. The quadratic residue symbol is then extended multiplicatively to all odd ideals $\nn$, and then also to all odd elements $\beta$ in $\OO_L$ by setting $(\alpha/\beta)_{L, 2} = (\alpha/\beta\OO_L)_{L, 2}$. To define the quartic residue symbol, we assume that $L$ contains $\Q(i)$. Then one can define the \textit{quartic residue symbol} $(\alpha/\pp)_{L, 4}$ as the element of $\{\pm 1, \pm i, 0\}$ such that
\[
\left(\frac{\alpha}{\pp}\right)_{L, 4} \equiv \alpha^{\frac{\Norm_{L/\Q}(\pp) - 1}{4}} \bmod \pp,
\]
and extend this to all odd ideals $\nn$ and odd elements $\beta$ in the same way as the quadratic residue symbol. A key property of the quartic residue symbol that we will use extensively is the following weak version of quartic reciprocity in $M := \Q(\zeta_8)$.
\begin{lemma}
\label{tQR}
Let $\alpha, \beta \in \OO_M$ with $\beta$ odd. Then $(\alpha/\beta)_{M, 4}$ depends only on the congruence class of $\beta$ modulo $16\alpha\OO_M$. Moreover, if $\alpha$ is also odd, then
$$
\left(\frac{\alpha}{\beta}\right)_{M, 4} = \mu \cdot \left(\frac{\beta}{\alpha}\right)_{M, 4},
$$
where $\mu\in\{\pm 1, \pm i\}$ depends only on the congruence classes of $\alpha$ and $\beta$ modulo $16\OO_{M}$.
\end{lemma}
\begin{proof}
This follows from \cite[Proposition 6.11, p.\ 199]{Lemmermeyer}.
\end{proof}

\subsection{Field lowering}
A key feature of our proof is the reduction of quartic residue symbols in a quartic number field to quadratic residue symbols in a quadratic field. We do this by using the following three lemmas.

\begin{lemma}
\label{lAlg1}
Let $K$ be a number field and let $\pp$ be an odd prime ideal of $K$. Suppose that $L$ is a quadratic extension of $K$ such that $L$ contains $\Q(i)$ and $\pp$ splits in $L$. Denote by $\psi$ the non-trivial element in $\text{Gal}(L/K)$. Then if $\psi$ fixes $\Q(i)$ we have for all $\alpha \in \OO_K$
\[
\left(\frac{\alpha}{\pp\OO_L}\right)_{L, 4} = \left(\frac{\alpha}{\pp\OO_K}\right)_{K, 2}
\]
and if $\psi$ does not fix $\Q(i)$ we have for all $\alpha \in \OO_K$ with $\pp \nmid \alpha$
\[
\left(\frac{\alpha}{\pp\OO_L}\right)_{L, 4} = 1
\]
\end{lemma}

\begin{proof}
Since $\pp$ splits in $L$, we can write $\pp = \mathfrak{q} \psi(\mathfrak{q})$ for some prime ideal $\mathfrak{q}$ of $L$. Hence we have
\[
\left(\frac{\alpha}{\pp\OO_L}\right)_{L, 4} = \left(\frac{\alpha}{\mathfrak{q}}\right)_{L, 4} \left(\frac{\alpha}{\psi(\mathfrak{q})}\right)_{L, 4}.
\]
If $\psi$ fixes $i$ we find that
\[
\left(\frac{\alpha}{\mathfrak{q}}\right)_{L, 4} = \psi \left(\left(\frac{\alpha}{\mathfrak{q}}\right)_{L, 4}\right) = \left(\frac{\psi(\alpha)}{\psi(\mathfrak{q})}\right)_{L, 4} = \left(\frac{\alpha}{\psi(\mathfrak{q})}\right)_{L, 4}.
\]
Combining this with the previous identity gives
\[
\left(\frac{\alpha}{\pp\OO_L}\right)_{L, 4} = \left(\frac{\alpha}{\mathfrak{q}}\right)_{L, 4}^2 = \left(\frac{\alpha}{\mathfrak{q}}\right)_{L, 2} = \left(\frac{\alpha}{\pp\OO_K}\right)_{K, 2},
\]
establishing the first part of the lemma. If $\psi$ does not fix $i$ we find that
\[
\left(\frac{\alpha}{\pp\OO_L}\right)_{L, 4} = \left(\frac{\alpha}{\mathfrak{q}}\right)_{L, 4} \left(\frac{\alpha}{\psi(\mathfrak{q})}\right)_{L, 4} = \left(\frac{\alpha}{\mathfrak{q}}\right)_{L, 4} \psi \left(\left(\frac{\alpha}{\mathfrak{q}}\right)_{L, 4}\right) = 1
\]
by checking this for all values of $(\alpha/\mathfrak{q})_{L, 4} \in \{\pm 1, \pm i\}$. This completes the proof.
\end{proof}

\begin{lemma}
\label{lAlg2}
Let $K$ be a number field and let $\pp$ be an odd prime ideal of $K$ of degree $1$ lying above $p$. Suppose that $L$ is a quadratic extension of $K$ such that $L$ contains $\Q(i)$ and $\pp$ stays inert in $L$. Then we have for all $\alpha \in \OO_K$
\[
\left(\frac{\alpha}{\pp\OO_L}\right)_{L, 4} = \left(\frac{\alpha}{\pp\OO_K}\right)_{K, 2}^{\frac{p + 1}{2}}.
\]
\end{lemma}

\begin{proof}
We have
\[
\left(\frac{\alpha}{\pp\OO_L}\right)_{L, 4} \equiv \alpha^{\frac{\Norm_L(\pp) - 1}{4}} \equiv \alpha^{\frac{p^2 - 1}{4}} \equiv \left(\alpha^{\frac{p - 1}{2}}\right)^\frac{p + 1}{2} \equiv \left(\alpha^{\frac{\Norm_K(\pp) - 1}{2}}\right)^\frac{p + 1}{2} \equiv \left(\frac{\alpha}{\pp\OO_K}\right)_{K, 2}^{\frac{p + 1}{2}} \bmod \pp,
\]
which immediately implies the lemma.
\end{proof}

\noindent Note that the previous lemmas only work if $\alpha \in \OO_K$. Our last lemma gives a way to ensure that $\alpha \in \OO_K$.

\begin{lemma}
\label{lAlg3}
Let $K$ be a number field and let $L$ be a quadratic extension of $K$. Denote by $\psi$ the non-trivial element in $\text{Gal}(L/K)$. Suppose that $\pp$ is a prime ideal of $K$ that does not ramify in $L$ and further suppose that $\beta \in \OO_L$ satisfies $\beta \equiv \psi(\beta) \bmod \pp \OO_L$. Then there is $\beta' \in \OO_K$ such that $\beta' \equiv \beta \bmod \pp\OO_L$.
\end{lemma}

\begin{proof}
Since by assumption $\pp$ does not ramify in $L$, we may assume that $\pp$ splits or stays inert in $L$. Let us first do the case that $\pp$ stays inert, which means precisely that $\psi(\pp) = \pp$. We conclude that $\psi$ is in the decomposition group of $\pp$. Furthermore, the inertia group of $\pp$ is trivial by the assumption that $\pp$ does not ramify. Since $\psi$ is not the identity, it follows that $\psi$ must become the Frobenius of the finite field extension $\OO_K/\pp \xhookrightarrow{} \OO_L/\pp$. Then $\beta \equiv \psi(\beta) \bmod \pp\OO_L$ means that $\beta$ is fixed by Frobenius. We conclude that $\beta$ comes from $\OO_K/\pp$, which we had to prove.

We still have to prove the lemma if $\pp$ splits. In this case we can write $\pp = \mathfrak{q} \psi(\mathfrak{q})$ for some prime ideal $\mathfrak{q}$ of $L$. Note that
\begin{align}
\label{eIsom}
\OO_K/\pp \xhookrightarrow{} \OO_L/\pp\OO_L \cong \OO_L/\mathfrak{q} \times \OO_L/\psi(\mathfrak{q}).
\end{align}
One checks that $\psi$ is the automorphism of $\OO_L/\mathfrak{q} \times \OO_L/\psi(\mathfrak{q})$ that maps $(x, y)$ to $(\psi(y), \psi(x))$. Hence $\beta \equiv \psi(\beta) \bmod \pp\OO_L$ implies that there is some $x \in \OO_L/\mathfrak{q}$ such that $\beta = (x, \psi(x))$ as an element of $\OO_L/\mathfrak{q} \times \OO_L/\psi(\mathfrak{q})$. Since $\OO_K/\pp \cong \OO_L/\mathfrak{q}$, we can pick $\beta' \in \OO_K$ such that $\beta'$ maps to $x$ under the natural inclusion $\OO_K/\pp \xhookrightarrow{} \OO_L/\mathfrak{q}$. Then it follows that $\beta$ maps to $(\beta', \psi(\beta'))$ under the maps given as in (\ref{eIsom}). This implies that $\beta' \equiv \beta \bmod \pp\OO_L$ as desired.
\end{proof}

\subsection{A fundamental domain for the action of $\OO_M^{\times}$}\label{sDomain}
In defining $a(\chi)_{\nn}$ for odd ideals $\nn$ of $\OO_M$, we had to choose a generator $w$ for the ideal $\nn$. There are many such choices, since the group of units of $\OO_M$ is quite large, i.e.,
$$
\OO_M^{\times} = \left\langle \zeta_8\right\rangle \times \left\langle \ve\right\rangle,
$$
where $\ve = 1+\sqrt{2}$ as before. It will be important to us that we can choose generators that are in some sense as small as possible. The following lemma allows us to do so.
\begin{lemma}
\label{fundDom}
There exists a subset $\DD$ of $\OO_M$ such that:
\begin{enumerate}
	\item $\DD$ is a fundamental domain for the natural action of $\OO_M^{\times}$ on $\OO_M$, i.e., for all $w\in\OO_M$, there exists a unique $n\in \Z$ such that $\ve^nw\in\DD$; and
	\item every non-zero ideal $\nn$ in $\OO_M$ has exactly $8$ generators in $\DD$; if $w$ is one such generator, then all such generators are of the form $\zeta w$, where $\zeta\in\left\langle \zeta_8\right\rangle$; and
	\item there exists a constant $C > 0$ such that for all $w = w_1+w_2\zeta_8 + w_3\zeta_8^2+w_4\zeta_8^3 \in \DD$ with $w_i\in\Z$, we have $|w_i|\leq C\cdot\Norm(\alpha)^{\frac{1}{4}}$.
\end{enumerate}
\end{lemma}
\begin{proof}
For a similar construction see \cite[Section 2.6, p.\ 12]{KM1}, which is based on \cite[Chapter 6, p.\ 158-181]{Marcus}.
\end{proof}

\subsection{The sieve}
We will prove Theorem~\ref{mainThm} by a sieve of Friedlander et al.\ \cite{FIMR} that generalizes the ideas of Vinogradov \cite{Vino1, Vino2} to the setting of number fields. Let $\chi$ be a Dirichlet character modulo $8$, and let $a(\chi)_{\nn}$ be defined as in \eqref{defachin}. We will prove the following two propositions.
\begin{prop}
\label{typeIprop}
There exists a real number $\theta_1 \in (0, 1)$ such that for every $\epsilon>0$, we have
$$
\sum_{\Norm(\nn)\leq X,\ \mm|\nn}a(\chi)_{\nn}\ll_{\epsilon} X^{1-\theta_1+\epsilon}
$$
uniformly for all non-zero ideals $\mm$ of $\OO_M$ and all $X\geq 2$. One can take $\theta_1 = \frac{1}{64}$.
\end{prop}

\begin{prop}
\label{typeIIprop}
There exists a real number $\theta_2 \in (0, 1)$ such that for every $\epsilon>0$, we have
$$
\sum_{\Norm(\mm)\leq M}\sum_{\Norm(\nn)\leq N}\alpha_{\mm}\beta_{\nn}a(\chi)_{\mm\nn}\ll_{\epsilon} (M+N)^{\theta_2}(MN)^{1-\theta_2+\epsilon}
$$
uniformly for all $M, N\geq 2$ and sequences of complex numbers $\{\alpha_{\mm}\}$ and $\{\beta_{\nn}\}$ satisfying $|\alpha_{\mm}|, |\beta_{\nn}|\leq 1$. One can take $\theta_2 = \frac{1}{24}$.
\end{prop}
Assuming Propositions \ref{typeIprop} and \ref{typeIIprop}, \cite[Proposition 5.2, p.\ 722]{FIMR} implies that there exists $\theta \in (0, 1)$ such that for every $\epsilon>0$, we have
$$
\sum_{\Norm(\nn)\leq X}a(\chi)_{\nn}\Lambda(\nn)\ll_{\epsilon} X^{1-\theta+\epsilon}
$$
uniformly for all $X\geq 2$. Moreover, we can take $\theta = 1/(49\cdot 64) = 1/3136$. By partial summation, it follows that, say,
\begin{equation}\label{psEst1}
\sum_{\Norm(\pp)\leq X}a(\chi)_{\pp}\ll X^{1-\frac{1}{3200}}.
\end{equation}
As
$$
\sum_{\substack{\Norm(\pp)\leq X \\ \pp\text{ lies over}\ p\not\equiv 1\bmod 8}} 1 \ll X^{\frac{1}{2}},
$$
Theorem~\ref{mainThm} follows from \eqref{psEst1} and Proposition~\ref{seqCrit}. It now remains to prove Propositions~\ref{typeIprop} and~\ref{typeIIprop}. 

\section{Proof of Proposition~\ref{typeIprop}}
\label{PROOFtypeIprop}
Let $\chi$ be a Dirichlet character modulo $8$. Let $\mathfrak{m}$ be an odd ideal of $\OO_M$. In view of Proposition \ref{seqCrit} we must bound the following sum
\[
A(x) = A(x; \chi, \mm) := \sum_{\substack{\Norm(\mathfrak{a}) \leq x \\ (\mathfrak{a}, 2) = 1,\ \mathfrak{m} \mid \mathfrak{a}}}\left([\alpha]_{\chi}+[\ve\alpha]_{\chi}\right),
\]
where $\alpha$ is chosen to be any generator of $\mathfrak{a}$. Our proof is based on the argument in \cite[Section 3, p.\ 12-19]{KM1}, which is in turn based on \cite[Section 6, p.\ 722-733]{FIMR}. Each non-zero ideal $\aaa$ has exactly $8$ generators $\alpha \in \DD$, where $\DD$ is the fundamental domain from Lemma~\ref{fundDom}. Set $u_1 = 1$ and $u_2 = \ve$. Set $F = 16$. Note that $\chi(\rat(\alpha\tau(\alpha)))$ depends only on the congruence class of $\alpha$ modulo~$8$. After splitting the above sum into congruence classes modulo $F$, and using \eqref{bracket2} and Lemma~\ref{tQR}, we find that
$$
A(x) = \frac{1}{8} \sum_{i = 1}^2 \sum_{\substack{\rho \bmod F\\ (\rho, F) = 1}} \mu(\rho, u_i)A(x; \rho, u_i),
$$
where $\mu(\rho, u_i)\in\{\pm 1, \pm i\}$ depends only on $\rho$ and $u_i$ and where
$$
A(x; \rho, u_i) := \sum_{\substack{\alpha \in u_i \DD,\ \Norm(\alpha) \leq x \\ \alpha \equiv \rho \bmod F \\ \alpha \equiv 0 \bmod \mathfrak{m}}} \left(\frac{\sigma(\alpha)}{\alpha}\right)_{M, 4} \left(\frac{\sigma\tau(\alpha)}{\alpha}\right)_{M, 4}.
$$
Our goal is to estimate $A(x; \rho, u_i)$ separately for each congruence class $\rho \bmod F$, $(\rho, F) = 1$ and unit $u_i$. We view $\OO_M$ as a $\Z$-module of rank $4$ and decompose it as $\OO_M = \Z \oplus \MM$, where $\MM = \Z\zeta_8\oplus\Z\zeta_8^2\oplus\Z\zeta_8^3$ is a free $\Z$-module of rank $3$. We can write $\alpha$ uniquely as
\[
\alpha = a + \beta, \text{ with } a \in \Z, \beta \in \MM,
\]
so that the summation conditions above are equivalent to
\begin{equation}\tag{$\ast$}
a + \beta \in u_i\DD,\quad \Norm(a + \beta) \leq x,\quad a + \beta \equiv \rho \bmod F,\quad a + \beta \equiv 0 \bmod \mathfrak{m}.
\end{equation}
We are now going to rewrite $(\sigma(\alpha)/\alpha)_{M, 4}$ and $(\sigma\tau(\alpha)/\alpha)_{M, 4}$ by using the same trick as in \cite[p.\ 725]{FIMR}. Put
$$
\sigma(\beta) - \beta = \eta^2 c_0 c \quad\text{ and }\quad\sigma\tau(\beta) - \beta = \eta'^2 c_0' c'
$$
with $c_0, c_0', c, c', \eta, \eta' \in \OO_M$, $c_0, c_0' \mid F$ squarefree, $\eta, \eta' \mid F^\infty$ and $(c, F) = (c', F) = 1$. By multiplying with an appropriate unit we can even ensure that $c \in \Z[i]$ and $c' \in \Z[\sqrt{-2}]$. Indeed, observe that
\begin{equation}\label{alphaprime}
\alpha' := \frac{\sigma(\alpha) - \alpha}{\zeta_8} = \frac{\sigma(\beta) - \beta}{\zeta_8} \in \Z[i],
\end{equation}
and we have a similar identity for $\sigma\tau(\beta) - \beta$. Then we obtain, just as in \cite[p.\ 14]{KM1}, by Lemma~\ref{tQR},
$$
\left(\frac{\sigma(\alpha)}{\alpha}\right)_{M, 4} = \mu_1 \cdot \left(\frac{a + \beta}{c\OO_M}\right)_{M, 4} \quad\text{ and }\quad\left(\frac{\sigma\tau(\alpha)}{\alpha}\right)_{M, 4} = \mu_2 \cdot \left(\frac{a + \beta}{c'\OO_M}\right)_{M, 4},
$$
where $\mu_1,\mu_2\in \{\pm 1, \pm i\}$ depend only on $\rho$ and $\beta$. Hence
$$
A(x; \rho, u_i) \leq \sum_{\beta \in \MM} |T(x; \beta, \rho, u_i)|,
$$
where
$$
T(x; \beta, \rho, u_i) := \sum_{\substack{a \in \Z \\ a+\beta\text{ sat. }(\ast)}} \left(\frac{a + \beta}{c\OO_M}\right)_{M, 4} \left(\frac{a + \beta}{c'\OO_M}\right)_{M, 4}.
$$
From now on we treat $\beta$ as fixed and estimate $T(x; \beta, \rho, u_i)$. It is here that we deviate from \cite{FIMR} and \cite{KM1}. Since we chose $c'\in\Z[\sqrt{-2}]$, we can factor the principal ideal $(c')\subset\Z[\sqrt{-2}]$ into prime ideals in $\Z[\sqrt{-2}]$ that do not ramify in $M$, say, $(c') = \prod_{i = 1}^k \pp_i^{e_i}$, so that
$$
\left(\frac{a + \beta}{c'\OO_M}\right)_{M, 4} = \prod_{i = 1}^k \left(\frac{a + \beta}{\pp_i \OO_M}\right)^{e_i}_{M, 4}.
$$
We claim that $((a + \beta)/\pp\OO_M)_{M, 4} = 1$ if $\pp \nmid a + \beta$. As a first step we can replace $\beta$ by some $\beta' \in \Z[\sqrt{-2}]$ due to Lemma~\ref{lAlg3}. Then Lemma~\ref{lAlg1} proves the claim if $\pp$ splits in $M$. Finally suppose that $\pp$ stays inert in $M$. If we define $p := \pp \cap \Z$, we find that $p \equiv 3 \bmod 8$. Hence Lemma~\ref{lAlg2} finishes the proof of the claim.

The factor $((a + \beta)/c\OO_M)_{M, 4}$ is handled more similarly to \cite[(6.21), p.\ 727]{FIMR}. Since we chose $c\in\Z[i]$, we factor $(c)\subset\Z[i]$ in $\Z[i]$ as $(c) = \mathfrak{g} \mathfrak{q}$, where $\mathfrak{g}$ is the product of all prime ideal powers $\pp^{e_{\pp}}\subset\Z[i]$ dividing $(c)$ such that the residue degree of $\pp$ is greater than one or such that $\pp$ is an unramified prime of degree one for which some conjugate of $\pp$ also divides $(c)$. By construction $q := N_{\Q(i)/\Q}(\mathfrak{q})$ is a squarefree odd integer and $g := N_{\Q(i)/\Q}(\mathfrak{g})$ is an odd squarefull number coprime with $q$.

Lemma~\ref{lAlg3} and the Chinese remainder theorem imply that there exists $\beta'\in\Z[i]$ such that $\beta\equiv \beta'\bmod \qq\OO_M$. Next, Lemma~\ref{lAlg1} and Lemma~\ref{lAlg2} imply that $((a+\beta')/\qq\OO_M)_{M, 4} = ((a+\beta')/\qq)_{\Q(i), 2}$. Finally, as $q$ is squarefree, the Chinese remainder theorem guarantees the existence of a rational integer~$b$ such that $\beta'\equiv b\bmod \qq$. Combining all of this gives
$$
\left(\frac{a + \beta}{c\OO_M}\right)_{M, 4} = \left(\frac{a + \beta}{\mathfrak{g}\OO_M}\right)_{M, 4} \left(\frac{a + b}{\mathfrak{q}}\right)_{\Q(i), 2}.
$$
Since $c$ depends on $\beta$ and not on $a$, we find that $b$ depends on $\beta$ and not on $a$. Now define $g_0$ as the radical of $g$, i.e., $
g_0 := \prod_{p \mid g} p$. We observe that the quartic residue symbol $(\alpha/\mathfrak{g}\OO_M)_{M, 4}$ is periodic in $\alpha$ modulo $\mathfrak{g}^\ast := \prod_{\pp \mid \mathfrak{g}} \pp$. But clearly $\mathfrak{g}^\ast$ divides $g_0$, and hence we conclude that $((a + \beta)/\mathfrak{g}\OO_M)_{M, 4}$ is periodic of period $g_0$ when viewed as a function of $a \in \Z$. So we split $T(x; \beta, \rho, u_i)$ into congruence classes modulo $g_0$, giving
$$
|T(x; \beta, \rho, u_i)| \leq \sum_{a_0 \bmod g_0} |T(x; \beta, \rho, u_i, a_0)|,
$$
where
$$
T(x; \beta, \rho, u_i, a_0) = \sum_{\substack{a \in \Z \\ a+\beta\text{ sat. }(\ast) \\ a \equiv a_0 \bmod g_0}} \left(\frac{a + b}{\mathfrak{q}}\right)_{\Q(i), 2} \left(\frac{a + \beta}{c'\OO_M}\right)_{M, 4}.
$$
We have already proven that $((a + \beta)/c'\OO_M)_{M, 4} = 1$ unless $\gcd(a + \beta, c') \neq (1)$ and in this case we have $((a + \beta)/c'\OO_M)_{M, 4} = 0$. An application of inclusion-exclusion gives
$$
|T(x; \beta, \rho, u_i, a_0)| \leq \sum_{\substack{\mathfrak{d} \mid c'\OO_M \\ \mathfrak{d} \text{ square-free}}} \left| T(x; \beta, \rho, u_i, a_0, \mathfrak{d}) \right|,
$$
where
\begin{equation}\label{eFin}
T(x; \beta, \rho, u_i, a_0, \mathfrak{d}) := \sum_{\substack{a \in \Z \\ a + \beta \text{ sat. }(\ast) \\ a \equiv a_0 \bmod g_0 \\ a + \beta \equiv 0 \bmod \mathfrak{d}}} \left(\frac{a + b}{\mathfrak{q}}\right)_{\Q(i), 2}.
\end{equation}
We unwrap the summation conditions above similarly as in \cite[p.\ 728]{FIMR}. Certainly $a + \beta \in u_i \DD$ implies that $a \ll x^{\frac{1}{4}}$, where the implied constant depends only on one of the two fixed units $u_i$. The condition $\Norm_{M/\Q}(a + \beta) \leq x$ is for fixed $\beta$ and $x$ a polynomial inequality of degree $4$ in $a$. Hence the summation variable $a \in \Z$ runs over at most $4$ intervals of length $\ll x^{1/4}$ with endpoints depending on $\beta$ and $x$.

Next, the congruence conditions $a + \beta \equiv \rho \bmod F$, $a + \beta \equiv 0 \bmod \mathfrak{m}$, $a \equiv a_0 \bmod g_0$ and $a + \beta \equiv 0 \bmod \mathfrak{d}$ imply that $a$ runs over some arithmetic progression of modulus $k$ dividing $g_0mdF$, where we define $m := \Norm_{M/\Q}(\mathfrak{m})$ and $d := \Norm_{M/\Q}(\mathfrak{d})$. Moreover, as $q = \Norm_{\Q(i)/\Q}(\mathfrak{q})$ is squarefree, $(\cdot/\mathfrak{q})_{\Q(i), 2}:\Z\rightarrow \{\pm 1, 0\}$ is the real primitive Dirichlet character of modulus $q$.

All in all, the sum in \eqref{eFin} can be rewritten as at most $4$ incomplete real character sums of length $\ll x^{\frac{1}{4}}$ and modulus $q \ll x^{\frac{1}{2}}$, each of which runs over an arithmetic progression of modulus $k$. When the modulus $q$ of the Dirichlet character divides the modulus $k$ of the arithmetic progression, one does not get the desired cancellation. So for now we assume that $q \nmid k$, and we will handle the case $q \mid k$ later. As has been explained in \cite[7., p.\ 924-925]{erratum}, Burgess's bound for short character sums \cite{Burgess} implies that for each integer $r\geq 2$, we have
$$
\displaystyle{T(x; \beta, \rho, u_i, a_0, \mathfrak{d}) \ll_{\epsilon, r} x^{\frac{1}{4}\left(1-\frac{1}{r}\right)}\cdot x^{\frac{1}{2}\left(\frac{r+1}{4r^2}+ \epsilon\right)}},
$$
so that on taking $r = 2$, we obtain
\begin{align}
\label{eBurgess}
T(x; \beta, \rho, u_i) \ll_{\epsilon} g_0x^{\frac{1}{4} - \frac{1}{32} + \epsilon}.
\end{align}
It remains to do the case $q \mid k$. Certainly, this implies $q \mid m d$. So (\ref{eBurgess}) holds if $q \nmid m d$. Recall that $(c) = \mathfrak{g} \mathfrak{q}$, hence we have (\ref{eBurgess}) unless
\begin{align}
\label{eExp}
p \mid \Norm_{\Q(i)/\Q}(\alpha') \implies p^2 \mid mdF\Norm_{\Q(i)/\Q}(\alpha'),
\end{align}
where $\alpha'$ is defined as in \eqref{alphaprime}. Define $A_\square(x; \rho, u_i)$ as the contribution to $A(x; \rho, u_i)$ with~\eqref{eExp}. Then we get
\[
A_\square(x; \rho, u_i) \leq |\{\alpha \in u_i \DD : N_{M/\Q}(\alpha) \leq x, \ p \mid \Norm_{\Q(i)/\Q}(\alpha') \implies p^2 \mid mdF\Norm_{\Q(i)/\Q}(\alpha')\}|.
\]
We decompose $\OO_M$ as $\OO_M = \Z[i] \oplus \MM'$, where $\MM' = \Z\zeta_8\oplus\Z\zeta_8^3 = \Z[i]\cdot\zeta_8$ is a free $\Z$-module of rank $2$. The linear map $\MM' \rightarrow \Z[i]$ given by $\alpha \mapsto \alpha'$ is injective. Now suppose $\alpha \in u_i \DD$ and $\Norm_{M/\Q}(\alpha) \leq x$. Then by Lemma~\ref{fundDom}, if we write $\alpha = a_1 + a_2i + (a_3 + a_4i)\zeta_8$, we have $a_j \ll x^{\frac{1}{4}}$ for $1\leq j\leq 4$. Hence the norm $\Norm_{\Q(i)/\Q}(\cdot)$ of $\alpha' = -2(a_3+a_4i)$ is $\ll x^{\frac{1}{2}}$, and so
$$
A_\square(x; \rho, u_i) \ll x^{\frac{1}{2}} |\{\alpha' \in \Z[i]: \Norm_{\Q(i)/\Q}(\alpha') \ll x^{\frac{1}{2}}, \ p \mid \Norm_{\Q(i)/\Q}(\alpha') \implies p^2 \mid mdF\Norm_{\Q(i)/\Q}(\alpha')\}|.
$$
Note that there are at most $b^{\epsilon}$ elements $\alpha' \in \Z[i]$ such that $\Norm_{\Q(i)/\Q}(\alpha') = b$. This gives
$$
A_\square(x; \rho, u_i) \ll x^{\frac{1}{2}+\epsilon} \sum_{\substack{b \ll x^{\frac{1}{2}}; \\ p \mid b \implies p^2 \mid mdFb}} 1,
$$
where $b$ runs over the positive rational integers. We assume that $m \leq x$ because otherwise $A(x)$ is the empty sum. This shows that $md \ll x^2$ and we conclude that
$$
A_\square(x; \rho, u_i) \ll_{\epsilon} x^{\frac{3}{4} + \epsilon}.
$$
Let $A_0(x; \rho, u_i)$ be the contribution to $A(x; \rho, u_i)$ of the terms $\alpha = a + \beta$ not satisfying (\ref{eExp}). Then we can split $A(x; \rho, u_i)$ as
$$
A(x; \rho, u_i) = A_\square(x; \rho, u_i) + A_0(x; \rho, u_i).
$$
To estimate $A_0(x; \rho, u_i)$ we can try to use our bound (\ref{eBurgess}) for every relevant $\beta$, but for this we need $g_0$ to be small. Hence we make the further partition
\[
A_0(x; \rho, u_i) = A_1(x; \rho, u_i) + A_2(x; \rho, u_i),
\]
where $\beta$ satisfies the additional constraint
\begin{align*}
g_0 \leq Z &\text{ in the sum } A_1(x; \rho, u_i), \\
g_0 > Z &\text{ in the sum } A_2(x; \rho, u_i).
\end{align*}
Here $Z$ is at our disposal, and we choose it later. We estimate $A_1(x; \rho, u_i)$ as in \cite{FIMR} by using (\ref{eBurgess}) and summing over $\beta \in \MM$ satisfying $|\beta^{(1)}|, \ldots, |\beta^{(4)}| \ll x^{\frac{1}{4}}$ to obtain
$$
A_1(x; \rho, u_i) \ll_{\epsilon} Zx^{1 - \frac{1}{32}+\epsilon}.
$$
To finish the proof of Proposition \ref{typeIprop} it remains to estimate $A_2(x; \rho, u_i)$. Note that $g_0 \leq \sqrt{g}$ and $g \leq \Norm_{\Q(i)/\Q}(c) \leq \Norm_{\Q(i)/\Q}(\alpha') \ll x^{\frac{1}{2}}$. Hence, similarly as for $A_\square(x; \rho, u_i)$, with $b = \Norm_{\Q(i)/\Q}(\alpha')$, we have
$$
A_2(x; \rho, u_i) \ll_{\epsilon} x^{\frac{1}{2}+\epsilon} \sum_{Z<g_0 \ll x^{\frac{1}{4}}}\sum_{\substack{b\ll x^{\frac{1}{2}} \\ g_0^2|b}}1 \ll_{\epsilon} Z^{-1}x^{1+\epsilon}.
$$
Picking $Z = x^{\frac{1}{64}}$ finishes the proof of Proposition \ref{typeIprop}.

\section{Proof of Proposition~\ref{typeIIprop}}\label{PROOFtypeIIprop}
Let $w$ and $z$ be odd elements in $\OO_M$. All quadratic and quartic residue symbols that follow are over $M$. By \eqref{bracket2}, we have
$$
[wz] = \left(\frac{\sigma(wz)\sigma\tau(wz)}{wz}\right)_4 = [w][z]\left(\frac{\sigma(w)}{z}\right)_4\left(\frac{\sigma\tau(w)}{z}\right)_4\left(\frac{\sigma(z)}{w}\right)_4\left(\frac{\sigma\tau(z)}{w}\right)_4.
$$
By Lemma~\ref{tQR}, we have, for some $\mu_1\in\{\pm 1, \pm i\}$ that depends only on the congruence classes of $w$ and $z$ modulo $16$,
$$
\left(\frac{\sigma(w)}{z}\right)_4\left(\frac{\sigma(z)}{w}\right)_4 = \mu_1\left(\frac{z}{\sigma(w)}\right)_4\left(\frac{\sigma(z)}{w}\right)_4 = \mu_1\left(\frac{z}{\sigma(w)}\right)_4\sigma\left(\frac{z}{\sigma(w)}\right)_4 = \mu_1\left(\frac{z}{\sigma(w)}\right)_2,
$$
because $\sigma(i) = i$. Similarly, for some $\mu_2\in\{\pm 1, \pm i\}$ that depends only on the congruence classes of $w$ and $z$ modulo $16$,
$$
\left(\frac{\sigma\tau(w)}{z}\right)_4\left(\frac{\sigma\tau(z)}{w}\right)_4 = \mu_2\left(\frac{z}{\sigma\tau(w)}\right)_4\sigma\tau\left(\frac{z}{\sigma\tau(w)}\right)_4 = \mu_2,
$$
because $\sigma\tau(i) = -i$. Hence we get, for $\mu_3 = \mu_1\mu_2$,
\begin{equation}\label{twistedmult}
[wz] = \mu_3 [w][z]\left(\frac{z}{\sigma(w)}\right)_2.
\end{equation}
This twisted multiplicativity formula for the symbol $[\cdot]$ is what makes the estimate in Proposition~\ref{typeIIprop} possible; it is analogous to \cite[Lemma 20.1, p.\ 1021]{FI1}, \cite[(3.8), p.\ 708]{FIMR}, \cite[Proposition 8, p.\ 31]{Milovic2}, and \cite[(4.1), p.\ 19]{KM1}. From here, the proof of Proposition~\ref{typeIIprop} is very similar to the proofs of \cite[Proposition 21.3, p.\ 1027]{FI1}, \cite[Proposition 7.1, p.\ 733]{FIMR}, \cite[Proposition 7, p.\ 29]{Milovic2}, and \cite[Proposition 2.6, p.\ 11]{KM1}. For the sake of completeness, we recall the main steps here. 

Let $\chi$ be a Dirichlet character modulo $8$, and let $\{a(\chi)_{\nn}\}_{\nn}$ be the sequence defined in \eqref{defachin}. Let $\{\alpha_{\mm}\}_{\mm}$ and $\{\beta_{\nn}\}_{\nn}$ be any two bounded sequences of complex numbers. Since each ideal of $\OO_M$ has $8$ different generators in $\DD$, we have
$$
\sum_{\Norm(\mm)\leq M}\sum_{\Norm(\nn)\leq N}\alpha_{\mm}\beta_{\nn}a(\chi)_{\mm\nn} = \frac{1}{8^2}\sum_{w\in\DD;\ \Norm(w)\leq M}\sum_{z\in\DD;\ \Norm(z)\leq N}\alpha_{w}\beta_{z}([wz]_{\chi}+[\ve wz ]_{\chi}).
$$
Here $\ve = 1+\sqrt{2}$, $\alpha_{w} := \alpha_{(w)}$ and $\beta_{z} := \beta_{(z)}$. Note that for any odd element $\alpha\in\OO_M$, we have $[\alpha]_{\chi} = \mu_4\cdot [\alpha]$ for some $\mu_4\in\{\pm 1, \pm i\}$ that depends only on the congruence class of $\alpha$ modulo $8$ (and so also modulo $16$). Also note that \eqref{twistedmult} implies that $[\ve wz] = \mu_5 [wz]$ for some $\mu_5\in\{\pm 1, \pm i\}$ that depends only on the congruence class of $wz$ modulo $16$. Hence, by restricting $w$ and $z$ to congruence classes modulo $16$, we may break up the sum above into $2\cdot 16^2$ sums of the shape
$$
\mu_6\sum_{\substack{w\in\DD;\ \Norm(w)\leq M \\ w\equiv \omega\bmod 16}}\sum_{\substack{z\in\DD;\ \Norm(z)\leq N \\ z\equiv \zeta\bmod 16}}\alpha_w\beta_z[wz],
$$
where $\mu_6\in\{\pm 1, \pm i\}$ depends only on the congruence classes $\omega$ and $\zeta$ modulo $16$. Again by \eqref{twistedmult}, we can replace $\alpha_w$ and $\beta_z$ by $\alpha_w[w]$ and $\beta_z[z]$ to arrive at the sum
$$
\mu_7\sum_{\substack{w\in\DD;\ \Norm(w)\leq M \\ w\equiv \omega\bmod 16}}\sum_{\substack{z\in\DD;\ \Norm(z)\leq N \\ z\equiv \zeta\bmod 16}}\alpha_w\beta_z\left(\frac{z}{\sigma(w)}\right)_2,
$$
where $\mu_7\in\{\pm 1, \pm i\}$ depends only on $\omega$ and $\zeta$. Hence proving Proposition~\ref{typeIIprop} reduces to proving the same estimate for sums of the type
\begin{equation}\label{BMN}
B(M, N; \omega, \zeta) := \sum_{\substack{w\in\DD;\ \Norm(w)\leq M \\ w\equiv \omega\bmod 16}}\sum_{\substack{z\in\DD;\ \Norm(z)\leq N \\ z\equiv \zeta\bmod 16}}\alpha_w\beta_z\left(\frac{z}{\sigma(w)}\right)_2.
\end{equation}
We will prove that
\begin{equation}\label{BMNbound}
B(M, N; \omega, \zeta) \ll_{\epsilon}M^{\frac{23}{24}}N(MN)^{\epsilon}
\end{equation}
whenever $N\geq M$; Proposition~\ref{typeIIprop} then immediately follows from the symmetry of the sum $B(M, N; \omega, \zeta)$ coming from quadratic reciprocity. So suppose that $N\geq M$. Similarly as in \cite{FI1, FIMR}, we fix an integer $k \geq 4$, and we apply H\"{o}lder's inequality to the $w$ variable to get
$$
|B(M, N; \omega, \zeta)|^{2k\cdot 2} \ll M^{(2k-1)\cdot 2}\left(\sum_{w}\left|\sum_{z}\beta_z\left(\frac{z}{\sigma(w)}\right)_2\right|^{2k}\right)^2,
$$
where the summations over $w$ and $z$ are as above in \eqref{BMN}. After expanding the inner sum in the second factor above, we get
$$
|B(M, N; \omega, \zeta)|^{4k} \ll M^{4k-2} \left( \sum_{w}\sum_{z}\beta_z' \left(\frac{z}{\sigma(w)}\right)_2\right)^2,
$$
where
$$
\beta_z' = \sum_{\substack{z = z_1\cdots z_{2k} \\ z_1,\ldots, z_{2k}\in \DD \\ \Norm(z_1),\ldots, \Norm(z_{2k})\leq N \\ z_1\equiv\cdots\equiv z_{2k}\equiv \zeta\bmod 16 }}\beta_{z_1}\overline{\beta_{z_2}}\cdots \beta_{z_{2k-1}}\overline{\beta_{z_{2k}}},
$$
and again the summation conditions for $w$ are as in \eqref{BMN}. Applying the Cauchy-Schwarz inequality to the $z$-variable above, we get
$$
|B(M, N; \omega, \zeta)|^{4k} \ll_{\epsilon, k} M^{4k-2} N^{2k+\epsilon} \sum_{w_1}\sum_{w_2}\sum_{z}\left(\frac{z}{\sigma(w_1w_2)}\right)_2,
$$
where the summation conditions for $w_1$ and $w_2$ are as those for $w$ in \eqref{BMN}, while the sum over $z$ is implicitly over those $z = z_1+z_2\zeta_8+z_3\zeta_8^2+z_4\zeta_8^3\in\Z[\zeta_8]$ with $|z_j|\ll N^{\frac{2k}{4}}$ for $1\leq j\leq 4$, the implied constant being absolute. We break up the sum over $z$ into congruence classes $\xi$ modulo $\Norm(w_1w_2)$ and note that 
$$
\sum_{\xi\bmod \sigma(w_1w_2)}\left(\frac{\xi}{\sigma(w_1w_2)}\right)_2 = 0
$$ 
unless $(\sigma(w_1w_2))$ is the square of an ideal in $\OO_M$, and in particular, unless $\Norm(w_1w_2)$ is a square. This gives
$$
\sum_{z}\left(\frac{z}{\sigma(w_1w_2)}\right)_2\ll
\begin{cases}
N^{2k} & \text{if }\Norm(w_1w_2)=\square \\
M^2N^{\frac{3k}{2}} + M^4N^k + M^6N^{\frac{k}{2}} + M^8 & \text{otherwise.}
\end{cases}
$$
Since we took $k\geq 4$ and since $N\geq M$, we have $N^{\frac{k}{2}}\geq M^2$, so the last bound can be simplified to $M^2N^{\frac{3k}{2}}$. Hence
$$
\begin{array}{rcl}
|B(M, N; \omega, \zeta)|^{4k} & \ll_{\epsilon, k} & M^{4k-2} N^{2k} \left(M\cdot N^{2k} + M^2\cdot M^2N^{\frac{3k}{2}}\right)(MN)^{\epsilon} \\
& \ll_{\epsilon, k} & \left(M^{4k-1} N^{4k} + M^{4k+2}N^{\frac{7k}{2}}\right)(MN)^{\epsilon}.
\end{array}
$$
Taking $k = 6$, we get
$$
|B(M, N; \omega, \zeta)| \ll_{\epsilon} \left(M^{\frac{23}{24}} N + M^{\frac{13}{12}}N^{\frac{7}{8}}\right)(MN)^{\epsilon}.
$$
Since $N\geq M$, we have $M^{\frac{13}{12}}N^{\frac{7}{8}} \leq M^{\frac{23}{24}} N$, and this finishes the proof of \eqref{BMNbound} and hence also of Proposition~\ref{typeIIprop}.
 
\bibliographystyle{plain}
\bibliography{Koymans_Milovic_2_References}

\begin{thebibliography}{10}

\bibitem{Burgess}
D.~A. Burgess.
\newblock On character sums and {$L$}-series. {II}.
\newblock {\em Proc. London Math. Soc. (3)}, 13:524--536, 1963.

\bibitem{CohnLag}
H.~Cohn and J.~C. Lagarias.
\newblock On the existence of fields governing the {$2$}-invariants of the
  classgroup of {${\bf Q}(\sqrt{dp})$} as {$p$} varies.
\newblock {\em Math. Comp.}, 41(164):711--730, 1983.

\bibitem{CohnLag2}
H.~Cohn and J.~C. Lagarias.
\newblock Is there a density for the set of primes {$p$} such that the class
  number of {${\bf Q}(\sqrt{-p})$} is divisible by {$16$}?
\newblock In {\em Topics in classical number theory, {V}ol. {I}, {II}
  ({B}udapest, 1981)}, volume~34 of {\em Colloq. Math. Soc. J\'anos Bolyai},
  pages 257--280. North-Holland, Amsterdam, 1984.

\bibitem{LVP}
C.-J. de~la Vall\'{e}e~Poussin.
\newblock {\em M\'{e}m. Couronn\'{e}s Acad. Roy. Belgique}, 59:1--74, 1899.

\bibitem{FK1}
{\'E}.~Fouvry and J.~Kl{\"u}ners.
\newblock Cohen-{L}enstra heuristics of quadratic number fields.
\newblock In {\em Algorithmic number theory}, volume 4076 of {\em Lecture Notes
  in Comput. Sci.}, pages 40--55. Springer, Berlin, 2006.

\bibitem{FK2}
{\'E}.~Fouvry and J.~Kl{\"u}ners.
\newblock On the 4-rank of class groups of quadratic number fields.
\newblock {\em Invent. Math.}, 167(3):455--513, 2007.

\bibitem{FK3}
{\'E}.~Fouvry and J.~Kl{\"u}ners.
\newblock On the negative {P}ell equation.
\newblock {\em Ann. of Math. (2)}, 172(3):2035--2104, 2010.

\bibitem{FK4}
{\'E}.~Fouvry and J.~Kl{\"u}ners.
\newblock The parity of the period of the continued fraction of {$\sqrt d$}.
\newblock {\em Proc. Lond. Math. Soc. (3)}, 101(2):337--391, 2010.

\bibitem{FI1}
J.~B. Friedlander and H.~Iwaniec.
\newblock The polynomial {$X^2+Y^4$} captures its primes.
\newblock {\em Ann. of Math. (2)}, 148(3):945--1040, 1998.

\bibitem{FIMR}
J.~B. Friedlander, H.~Iwaniec, B.~Mazur, and K.~Rubin.
\newblock The spin of prime ideals.
\newblock {\em Invent. Math.}, 193(3):697--749, 2013.

\bibitem{erratum}
J.~B. Friedlander, H.~Iwaniec, B.~Mazur, and K.~Rubin.
\newblock Erratum to: {T}he spin of prime ideals [ {MR}3091978].
\newblock {\em Invent. Math.}, 202(2):923--925, 2015.

\bibitem{Gauss}
C.~F. Gauss.
\newblock {\em Disquisitiones arithmeticae}.
\newblock Springer-Verlag, New York, 1986.
\newblock Translated and with a preface by Arthur A. Clarke, Revised by William
  C. Waterhouse, Cornelius Greither and A. W. Grootendorst and with a preface
  by Waterhouse.

\bibitem{KM1}
P.~{Koymans} and D.~{Milovic}.
\newblock {On the $16$-rank of class groups of $\mathbb{Q}(\sqrt{-p})$}.
\newblock {\em ArXiv e-prints}, November 2016.

\bibitem{Lemmermeyer}
F.~Lemmermeyer.
\newblock {\em Reciprocity laws}.
\newblock Springer Monographs in Mathematics. Springer-Verlag, Berlin, 2000.
\newblock From Euler to Eisenstein.

\bibitem{LW82}
P.~A. Leonard and K.~S. Williams.
\newblock On the divisibility of the class numbers of {$Q(\sqrt{-p})$} and
  {$Q(\sqrt{-2p})$} by {$16$}.
\newblock {\em Canad. Math. Bull.}, 25(2):200--206, 1982.

\bibitem{Marcus}
D.~A. Marcus.
\newblock {\em Number fields}.
\newblock Springer-Verlag, New York-Heidelberg, 1977.
\newblock Universitext.

\bibitem{Milovic2}
D.~{Milovic}.
\newblock {On the $16$-rank of class groups of $\mathbb{Q}(\sqrt{-8p})$ for
  $p\equiv -1\bmod 4$}.
\newblock {\em ArXiv e-prints}, November 2015.

\bibitem{Morton}
P.~Morton.
\newblock Density result for the {$2$}-classgroups of imaginary quadratic
  fields.
\newblock {\em J. Reine Angew. Math.}, 332:156--187, 1982.

\bibitem{Redei}
L.~R{\'e}dei.
\newblock Arithmetischer {B}eweis des {S}atzes \"uber die {A}nzahl der durch
  vier teilbaren {I}nvarianten der absoluten {K}lassengruppe im quadratischen
  {Z}ahlk\"orper.
\newblock {\em J. Reine Angew. Math.}, 171:55--60, 1934.

\bibitem{Smith1}
A.~{Smith}.
\newblock {Governing fields and statistics for 4-Selmer groups and 8-class
  groups}.
\newblock {\em ArXiv e-prints}, July 2016.

\bibitem{Smith2}
A.~{Smith}.
\newblock {$2^\infty$-Selmer groups, $2^\infty$-class groups, and Goldfeld's
  conjecture}.
\newblock {\em ArXiv e-prints}, February 2017.

\bibitem{Ste1}
P.~Stevenhagen.
\newblock Ray class groups and governing fields.
\newblock In {\em Th\'eorie des nombres, {A}nn\'ee 1988/89, {F}asc.\ 1}, Publ.
  Math. Fac. Sci. Besan\c con, page~93. Univ. Franche-Comt\'e, Besan\c con,
  1989.

\bibitem{Ste2}
P.~Stevenhagen.
\newblock Divisibility by {$2$}-powers of certain quadratic class numbers.
\newblock {\em J. Number Theory}, 43(1):1--19, 1993.

\bibitem{Vaughan}
R.-C. Vaughan.
\newblock Sommes trigonom\'etriques sur les nombres premiers.
\newblock {\em C. R. Acad. Sci. Paris S\'er. A-B}, 285(16):A981--A983, 1977.

\bibitem{Vino1}
I.~M. Vinogradov.
\newblock The method of trigonometrical sums in the theory of numbers.
\newblock {\em Trav. Inst. Math. Stekloff}, 23:109, 1947.

\bibitem{Vino2}
I.~M. Vinogradov.
\newblock {\em The method of trigonometrical sums in the theory of numbers}.
\newblock Dover Publications, Inc., Mineola, NY, 2004.
\newblock Translated from the Russian, revised and annotated by K. F. Roth and
  Anne Davenport, Reprint of the 1954 translation.

\end{thebibliography}
\end{document}